\documentclass[12pt]{article}
\usepackage{amssymb}
\usepackage{amsfonts}
\usepackage{german}
\long\def\ig#1{\relax}
\ig{Thanks to Roberto Minio for this def'n.  Compare the def'n of
\comment in AMSTeX.}
 
\newcount \coefa
\newcount \coefb
\newcount \coefc
\newcount\tempcounta
\newcount\tempcountb
\newcount\tempcountc
\newcount\tempcountd
\newcount\xext
\newcount\yext
\newcount\xoff
\newcount\yoff
\newcount\gap%
\newcount\arrowtypea
\newcount\arrowtypeb
\newcount\arrowtypec
\newcount\arrowtyped
\newcount\arrowtypee
\newcount\height
\newcount\width
\newcount\xpos
\newcount\ypos
\newcount\run
\newcount\rise
\newcount\arrowlength
\newcount\halflength
\newcount\arrowtype
\newdimen\tempdimen
\newdimen\xlen
\newdimen\ylen
\newsavebox{\tempboxa}%
\newsavebox{\tempboxb}%
\newsavebox{\tempboxc}%
 
\makeatletter
\def\settypes(#1,#2,#3){\arrowtypea#1 \arrowtypeb#2 \arrowtypec#3}
\def\settoheight#1#2{\setbox\@tempboxa\hbox{#2}#1\ht\@tempboxa\relax}%
\def\settodepth#1#2{\setbox\@tempboxa\hbox{#2}#1\dp\@tempboxa\relax}%
\def\settokens[#1`#2`#3`#4]{%
     \def\tokena{#1}\def\tokenb{#2}\def\tokenc{#3}\def\tokend{#4}}
\def\setsqparms[#1`#2`#3`#4;#5`#6]{%
\arrowtypea #1
\arrowtypeb #2
\arrowtypec #3
\arrowtyped #4
\width #5
\height #6
}
\def\setpos(#1,#2){\xpos=#1 \ypos#2}
 
\def\bfig{\begin{picture}(\xext,\yext)(\xoff,\yoff)}
\def\efig{\end{picture}}
 
\def\putbox(#1,#2)#3{\put(#1,#2){\makebox(0,0){$#3$}}}
 
\def\settriparms[#1`#2`#3;#4]{\settripairparms[#1`#2`#3`1`1;#4]}%
 
\def\settripairparms[#1`#2`#3`#4`#5;#6]{%
\arrowtypea #1
\arrowtypeb #2
\arrowtypec #3
\arrowtyped #4
\arrowtypee #5
\width #6
\height #6
}
 
\def\resetparms{\settripairparms[1`1`1`1`1;500]\width 500}
 
\resetparms
 
\def\mvector(#1,#2)#3{
\put(0,0){\vector(#1,#2){#3}}%
\put(0,0){\vector(#1,#2){30}}%
}
\def\evector(#1,#2)#3{{
\arrowlength #3
\put(0,0){\vector(#1,#2){\arrowlength}}%
\advance \arrowlength by-30
\put(0,0){\vector(#1,#2){\arrowlength}}%
}}
 
\def\horsize#1#2{%
\settowidth{\tempdimen}{$#2$}%
#1=\tempdimen
\divide #1 by\unitlength
}
 
\def\vertsize#1#2{%
\settoheight{\tempdimen}{$#2$}%
#1=\tempdimen
\settodepth{\tempdimen}{$#2$}%
\advance #1 by\tempdimen
\divide #1 by\unitlength
}
 
\def\vertadjust[#1`#2`#3]{%
\vertsize{\tempcounta}{#1}%
\vertsize{\tempcountb}{#2}%
\ifnum \tempcounta<\tempcountb \tempcounta=\tempcountb \fi
\divide\tempcounta by2
\vertsize{\tempcountb}{#3}%
\ifnum \tempcountb>0 \advance \tempcountb by20 \fi
\ifnum \tempcounta<\tempcountb \tempcounta=\tempcountb \fi
}
 
\def\horadjust[#1`#2`#3]{%
\horsize{\tempcounta}{#1}%
\horsize{\tempcountb}{#2}%
\ifnum \tempcounta<\tempcountb \tempcounta=\tempcountb \fi
\divide\tempcounta by20
\horsize{\tempcountb}{#3}%
\ifnum \tempcountb>0 \advance \tempcountb by60 \fi
\ifnum \tempcounta<\tempcountb \tempcounta=\tempcountb \fi
}
 
\ig{ In this procedure, #1 is the paramater that sticks out all the way,
#2 sticks out the least and #3 is a label sticking out half way.  #4 is
the amount of the offset.}
 
\def\sladjust[#1`#2`#3]#4{%
\tempcountc=#4
\horsize{\tempcounta}{#1}%
\divide \tempcounta by2
\horsize{\tempcountb}{#2}%
\divide \tempcountb by2
\advance \tempcountb by-\tempcountc
\ifnum \tempcounta<\tempcountb \tempcounta=\tempcountb\fi
\divide \tempcountc by2
\horsize{\tempcountb}{#3}%
\advance \tempcountb by-\tempcountc
\ifnum \tempcountb>0 \advance \tempcountb by80\fi
\ifnum \tempcounta<\tempcountb \tempcounta=\tempcountb\fi
\advance\tempcounta by20
}
 
\def\putvector(#1,#2)(#3,#4)#5#6{{%
\xpos=#1
\ypos=#2
\run=#3
\rise=#4
\arrowlength=#5
\arrowtype=#6
\ifnum \arrowtype<0
    \ifnum \run=0
        \advance \ypos by-\arrowlength
    \else
        \tempcounta \arrowlength
        \multiply \tempcounta by\rise
        \divide \tempcounta by\run
        \ifnum\run>0
            \advance \xpos by\arrowlength
            \advance \ypos by\tempcounta
        \else
            \advance \xpos by-\arrowlength
            \advance \ypos by-\tempcounta
        \fi
    \fi
    \multiply \arrowtype by-1
    \multiply \rise by-1
    \multiply \run by-1
\fi
\ifnum \arrowtype=1
    \put(\xpos,\ypos){\vector(\run,\rise){\arrowlength}}%
\else\ifnum \arrowtype=2
    \put(\xpos,\ypos){\mvector(\run,\rise)\arrowlength}%
\else\ifnum\arrowtype=3
    \put(\xpos,\ypos){\evector(\run,\rise){\arrowlength}}%
\fi\fi\fi
}}
 
\def\putsplitvector(#1,#2)#3#4{
\xpos #1
\ypos #2
\arrowtype #4
\halflength #3
\arrowlength #3
\gap 140
\advance \halflength by-\gap
\divide \halflength by2
\ifnum \arrowtype=1
    \put(\xpos,\ypos){\line(0,-1){\halflength}}%
    \advance\ypos by-\halflength
    \advance\ypos by-\gap
    \put(\xpos,\ypos){\vector(0,-1){\halflength}}%
\else\ifnum \arrowtype=2
    \put(\xpos,\ypos){\line(0,-1)\halflength}%
    \put(\xpos,\ypos){\vector(0,-1)3}%
    \advance\ypos by-\halflength
    \advance\ypos by-\gap
    \put(\xpos,\ypos){\vector(0,-1){\halflength}}%
\else\ifnum\arrowtype=3
    \put(\xpos,\ypos){\line(0,-1)\halflength}%
    \advance\ypos by-\halflength
    \advance\ypos by-\gap
    \put(\xpos,\ypos){\evector(0,-1){\halflength}}%
\else\ifnum \arrowtype=-1
    \advance \ypos by-\arrowlength
    \put(\xpos,\ypos){\line(0,1){\halflength}}%
    \advance\ypos by\halflength
    \advance\ypos by\gap
    \put(\xpos,\ypos){\vector(0,1){\halflength}}%
\else\ifnum \arrowtype=-2
    \advance \ypos by-\arrowlength
    \put(\xpos,\ypos){\line(0,1)\halflength}%
    \put(\xpos,\ypos){\vector(0,1)3}%
    \advance\ypos by\halflength
    \advance\ypos by\gap
    \put(\xpos,\ypos){\vector(0,1){\halflength}}%
\else\ifnum\arrowtype=-3
    \advance \ypos by-\arrowlength
    \put(\xpos,\ypos){\line(0,1)\halflength}%
    \advance\ypos by\halflength
    \advance\ypos by\gap
    \put(\xpos,\ypos){\evector(0,1){\halflength}}%
\fi\fi\fi\fi\fi\fi
}
 
\def\putmorphism(#1)(#2,#3)[#4`#5`#6]#7#8#9{{%
\run #2
\rise #3
\ifnum\rise=0
  \puthmorphism(#1)[#4`#5`#6]{#7}{#8}{#9}%
\else\ifnum\run=0
  \putvmorphism(#1)[#4`#5`#6]{#7}{#8}{#9}%
\else
\setpos(#1)%
\arrowlength #7
\arrowtype #8
\ifnum\run=0
\else\ifnum\rise=0
\else
\ifnum\run>0
    \coefa=1
\else
   \coefa=-1
\fi
\ifnum\arrowtype>0
   \coefb=0
   \coefc=-1
\else
   \coefb=\coefa
   \coefc=1
   \arrowtype=-\arrowtype
\fi
\width=2
\multiply \width by\run
\divide \width by\rise
\ifnum \width<0  \width=-\width\fi
\advance\width by60
\if l#9 \width=-\width\fi
\putbox(\xpos,\ypos){#4}
{\multiply \coefa by\arrowlength
\advance\xpos by\coefa
\multiply \coefa by\rise
\divide \coefa by\run
\advance \ypos by\coefa
\putbox(\xpos,\ypos){#5} }%
{\multiply \coefa by\arrowlength
\divide \coefa by2
\advance \xpos by\coefa
\advance \xpos by\width
\multiply \coefa by\rise
\divide \coefa by\run
\advance \ypos by\coefa
\if l#9%
   \put(\xpos,\ypos){\makebox(0,0)[r]{$#6$}}%
\else\if r#9%
   \put(\xpos,\ypos){\makebox(0,0)[l]{$#6$}}%
\fi\fi }%
{\multiply \rise by-\coefc
\multiply \run by-\coefc
\multiply \coefb by\arrowlength
\advance \xpos by\coefb
\multiply \coefb by\rise
\divide \coefb by\run
\advance \ypos by\coefb
\multiply \coefc by70
\advance \ypos by\coefc
\multiply \coefc by\run
\divide \coefc by\rise
\advance \xpos by\coefc
\multiply \coefa by140
\multiply \coefa by\run
\divide \coefa by\rise
\advance \arrowlength by\coefa
\ifnum \arrowtype=1
   \put(\xpos,\ypos){\vector(\run,\rise){\arrowlength}}%
\else\ifnum\arrowtype=2
   \put(\xpos,\ypos){\mvector(\run,\rise){\arrowlength}}%
\else\ifnum\arrowtype=3
   \put(\xpos,\ypos){\evector(\run,\rise){\arrowlength}}%
\fi\fi\fi}\fi\fi\fi\fi}}
 
\def\puthmorphism(#1,#2)[#3`#4`#5]#6#7#8{{%
\xpos #1
\ypos #2
\width #6
\arrowlength #6
\putbox(\xpos,\ypos){#3\vphantom{#4}}%
{\advance \xpos by\arrowlength
\putbox(\xpos,\ypos){\vphantom{#3}#4}}%
\horsize{\tempcounta}{#3}%
\horsize{\tempcountb}{#4}%
\divide \tempcounta by2
\divide \tempcountb by2
\advance \tempcounta by30
\advance \tempcountb by30
\advance \xpos by\tempcounta
\advance \arrowlength by-\tempcounta
\advance \arrowlength by-\tempcountb
\putvector(\xpos,\ypos)(1,0){\arrowlength}{#7}%
\divide \arrowlength by2
\advance \xpos by\arrowlength
\vertsize{\tempcounta}{#5}%
\divide\tempcounta by2
\advance \tempcounta by20
\if a#8 %
   \advance \ypos by\tempcounta
   \putbox(\xpos,\ypos){#5}%
\else
   \advance \ypos by-\tempcounta
   \putbox(\xpos,\ypos){#5}%
\fi}}
 
\def\putvmorphism(#1,#2)[#3`#4`#5]#6#7#8{{%
\xpos #1
\ypos #2
\arrowlength #6
\arrowtype #7
\settowidth{\xlen}{$#5$}%
\putbox(\xpos,\ypos){#3}%
{\advance \ypos by-\arrowlength
\putbox(\xpos,\ypos){#4}}%
{\advance\arrowlength by-140
\advance \ypos by-70
\ifdim\xlen>0pt
   \if m#8%
      \putsplitvector(\xpos,\ypos){\arrowlength}{\arrowtype}%
   \else
      \putvector(\xpos,\ypos)(0,-1){\arrowlength}{\arrowtype}%
   \fi
\else
   \putvector(\xpos,\ypos)(0,-1){\arrowlength}{\arrowtype}%
\fi}%
\ifdim\xlen>0pt
   \divide \arrowlength by2
   \advance\ypos by-\arrowlength
   \if l#8%
      \advance \xpos by-40
      \put(\xpos,\ypos){\makebox(0,0)[r]{$#5$}}%
   \else\if r#8%
      \advance \xpos by40
      \put(\xpos,\ypos){\makebox(0,0)[l]{$#5$}}%
   \else
      \putbox(\xpos,\ypos){#5}%
   \fi\fi
\fi
}}
 
\def\topadjust[#1`#2`#3]{%
\yoff=10
\vertadjust[#1`#2`{#3}]%
\advance \yext by\tempcounta
\advance \yext by 10
}
\def\botadjust[#1`#2`#3]{%
\vertadjust[#1`#2`{#3}]%
\advance \yext by\tempcounta
\advance \yoff by-\tempcounta
}
\def\leftadjust[#1`#2`#3]{%
\xoff=0
\horadjust[#1`#2`{#3}]%
\advance \xext by\tempcounta
\advance \xoff by-\tempcounta
}
\def\rightadjust[#1`#2`#3]{%
\horadjust[#1`#2`{#3}]%
\advance \xext by\tempcounta
}
\def\rightsladjust[#1`#2`#3]{%
\sladjust[#1`#2`{#3}]{\width}%
\advance \xext by\tempcounta
}
\def\leftsladjust[#1`#2`#3]{%
\xoff=0
\sladjust[#1`#2`{#3}]{\width}%
\advance \xext by\tempcounta
\advance \xoff by-\tempcounta
}
\def\adjust[#1`#2;#3`#4;#5`#6;#7`#8]{%
\topadjust[#1``{#2}]
\leftadjust[#3``{#4}]
\rightadjust[#5``{#6}]
\botadjust[#7``{#8}]}
 
\def\putsquarep<#1>(#2)[#3;#4`#5`#6`#7]{{%
\setsqparms[#1]%
\setpos(#2)%
\settokens[#3]%
\puthmorphism(\xpos,\ypos)[\tokenc`\tokend`{#7}]{\width}{\arrowtyped}b%
\advance\ypos by \height
\puthmorphism(\xpos,\ypos)[\tokena`\tokenb`{#4}]{\width}{\arrowtypea}a%
\putvmorphism(\xpos,\ypos)[``{#5}]{\height}{\arrowtypeb}l%
\advance\xpos by \width
\putvmorphism(\xpos,\ypos)[``{#6}]{\height}{\arrowtypec}r%
}}
 
\def\putsquare{\@ifnextchar <{\putsquarep}{\putsquarep%
   <\arrowtypea`\arrowtypeb`\arrowtypec`\arrowtyped;\width`\height>}}
\def\square{\@ifnextchar< {\squarep}{\squarep
   <\arrowtypea`\arrowtypeb`\arrowtypec`\arrowtyped;\width`\height>}}
\def\squarep<#1>[#2`#3`#4`#5;#6`#7`#8`#9]{{
\setsqparms[#1]
\xext=\width                                          
\yext=\height                                         
\topadjust[#2`#3`{#6}]
\botadjust[#4`#5`{#9}]
\leftadjust[#2`#4`{#7}]
\rightadjust[#3`#5`{#8}]
\begin{picture}(\xext,\yext)(\xoff,\yoff)
\putsquarep<\arrowtypea`\arrowtypeb`\arrowtypec`\arrowtyped;\width`\height>%
(0,0)[#2`#3`#4`#5;#6`#7`#8`{#9}]%
\end{picture}%
}}
 
\def\putptrianglep<#1>(#2,#3)[#4`#5`#6;#7`#8`#9]{{%
\settriparms[#1]%
\xpos=#2 \ypos=#3
\advance\ypos by \height
\puthmorphism(\xpos,\ypos)[#4`#5`{#7}]{\height}{\arrowtypea}a%
\putvmorphism(\xpos,\ypos)[`#6`{#8}]{\height}{\arrowtypeb}l%
\advance\xpos by\height
\putmorphism(\xpos,\ypos)(-1,-1)[``{#9}]{\height}{\arrowtypec}r%
}}
 
\def\putptriangle{\@ifnextchar <{\putptrianglep}{\putptrianglep
   <\arrowtypea`\arrowtypeb`\arrowtypec;\height>}}
\def\ptriangle{\@ifnextchar <{\ptrianglep}{\ptrianglep
   <\arrowtypea`\arrowtypeb`\arrowtypec;\height>}}
 
\def\ptrianglep<#1>[#2`#3`#4;#5`#6`#7]{{
\settriparms[#1]%
\width=\height                         
\xext=\width                           
\yext=\width                           
\topadjust[#2`#3`{#5}]
\botadjust[#3``]
\leftadjust[#2`#4`{#6}]
\rightsladjust[#3`#4`{#7}]
\begin{picture}(\xext,\yext)(\xoff,\yoff)
\putptrianglep<\arrowtypea`\arrowtypeb`\arrowtypec;\height>%
(0,0)[#2`#3`#4;#5`#6`{#7}]%
\end{picture}%
}}
 
\def\putqtrianglep<#1>(#2,#3)[#4`#5`#6;#7`#8`#9]{{%
\settriparms[#1]%
\xpos=#2 \ypos=#3
\advance\ypos by\height
\puthmorphism(\xpos,\ypos)[#4`#5`{#7}]{\height}{\arrowtypea}a%
\putmorphism(\xpos,\ypos)(1,-1)[``{#8}]{\height}{\arrowtypeb}l%
\advance\xpos by\height
\putvmorphism(\xpos,\ypos)[`#6`{#9}]{\height}{\arrowtypec}r%
}}
 
\def\putqtriangle{\@ifnextchar <{\putqtrianglep}{\putqtrianglep
   <\arrowtypea`\arrowtypeb`\arrowtypec;\height>}}
\def\qtriangle{\@ifnextchar <{\qtrianglep}{\qtrianglep
   <\arrowtypea`\arrowtypeb`\arrowtypec;\height>}}
 
\def\qtrianglep<#1>[#2`#3`#4;#5`#6`#7]{{
\settriparms[#1]
\width=\height                         
\xext=\width                           
\yext=\height                          
\topadjust[#2`#3`{#5}]
\botadjust[#4``]
\leftsladjust[#2`#4`{#6}]
\rightadjust[#3`#4`{#7}]
\begin{picture}(\xext,\yext)(\xoff,\yoff)
\putqtrianglep<\arrowtypea`\arrowtypeb`\arrowtypec;\height>%
(0,0)[#2`#3`#4;#5`#6`{#7}]%
\end{picture}%
}}
 
\def\putdtrianglep<#1>(#2,#3)[#4`#5`#6;#7`#8`#9]{{%
\settriparms[#1]%
\xpos=#2 \ypos=#3
\puthmorphism(\xpos,\ypos)[#5`#6`{#9}]{\height}{\arrowtypec}b%
\advance\xpos by \height \advance\ypos by\height
\putmorphism(\xpos,\ypos)(-1,-1)[``{#7}]{\height}{\arrowtypea}l%
\putvmorphism(\xpos,\ypos)[#4``{#8}]{\height}{\arrowtypeb}r%
}}
 
\def\putdtriangle{\@ifnextchar <{\putdtrianglep}{\putdtrianglep
   <\arrowtypea`\arrowtypeb`\arrowtypec;\height>}}
\def\dtriangle{\@ifnextchar <{\dtrianglep}{\dtrianglep
   <\arrowtypea`\arrowtypeb`\arrowtypec;\height>}}
 
\def\dtrianglep<#1>[#2`#3`#4;#5`#6`#7]{{
\settriparms[#1]
\width=\height                         
\xext=\width                           
\yext=\height                          
\topadjust[#2``]
\botadjust[#3`#4`{#7}]
\leftsladjust[#3`#2`{#5}]
\rightadjust[#2`#4`{#6}]
\begin{picture}(\xext,\yext)(\xoff,\yoff)
\putdtrianglep<\arrowtypea`\arrowtypeb`\arrowtypec;\height>%
(0,0)[#2`#3`#4;#5`#6`{#7}]%
\end{picture}%
}}
 
\def\putbtrianglep<#1>(#2,#3)[#4`#5`#6;#7`#8`#9]{{%
\settriparms[#1]%
\xpos=#2 \ypos=#3
\puthmorphism(\xpos,\ypos)[#5`#6`{#9}]{\height}{\arrowtypec}b%
\advance\ypos by\height
\putmorphism(\xpos,\ypos)(1,-1)[``{#8}]{\height}{\arrowtypeb}r%
\putvmorphism(\xpos,\ypos)[#4``{#7}]{\height}{\arrowtypea}l%
}}
 
\def\putbtriangle{\@ifnextchar <{\putbtrianglep}{\putbtrianglep
   <\arrowtypea`\arrowtypeb`\arrowtypec;\height>}}
\def\btriangle{\@ifnextchar <{\btrianglep}{\btrianglep
   <\arrowtypea`\arrowtypeb`\arrowtypec;\height>}}
 
\def\btrianglep<#1>[#2`#3`#4;#5`#6`#7]{{
\settriparms[#1]
\width=\height                         
\xext=\width                           
\yext=\height                          
\topadjust[#2``]
\botadjust[#3`#4`{#7}]
\leftadjust[#2`#3`{#5}]
\rightsladjust[#4`#2`{#6}]
\begin{picture}(\xext,\yext)(\xoff,\yoff)
\putbtrianglep<\arrowtypea`\arrowtypeb`\arrowtypec;\height>%
(0,0)[#2`#3`#4;#5`#6`{#7}]%
\end{picture}%
}}
 
\def\putAtrianglep<#1>(#2,#3)[#4`#5`#6;#7`#8`#9]{{%
\settriparms[#1]%
\xpos=#2 \ypos=#3
{\multiply \height by2
\puthmorphism(\xpos,\ypos)[#5`#6`{#9}]{\height}{\arrowtypec}b}%
\advance\xpos by\height \advance\ypos by\height
\putmorphism(\xpos,\ypos)(-1,-1)[#4``{#7}]{\height}{\arrowtypea}l%
\putmorphism(\xpos,\ypos)(1,-1)[``{#8}]{\height}{\arrowtypeb}r%
}}
 
\def\putAtriangle{\@ifnextchar <{\putAtrianglep}{\putAtrianglep
   <\arrowtypea`\arrowtypeb`\arrowtypec;\height>}}
\def\Atriangle{\@ifnextchar <{\Atrianglep}{\Atrianglep
   <\arrowtypea`\arrowtypeb`\arrowtypec;\height>}}
 
\def\Atrianglep<#1>[#2`#3`#4;#5`#6`#7]{{
\settriparms[#1]
\width=\height                         
\xext=\width                           
\yext=\height                          
\topadjust[#2``]
\botadjust[#3`#4`{#7}]
\multiply \xext by2 
\leftsladjust[#3`#2`{#5}]
\rightsladjust[#4`#2`{#6}]
\begin{picture}(\xext,\yext)(\xoff,\yoff)%
\putAtrianglep<\arrowtypea`\arrowtypeb`\arrowtypec;\height>%
(0,0)[#2`#3`#4;#5`#6`{#7}]%
\end{picture}%
}}
 
\def\putAtrianglepairp<#1>(#2)[#3;#4`#5`#6`#7`#8]{{
\settripairparms[#1]%
\setpos(#2)%
\settokens[#3]%
\puthmorphism(\xpos,\ypos)[\tokenb`\tokenc`{#7}]{\height}{\arrowtyped}b%
\advance\xpos by\height
\advance\ypos by\height
\putmorphism(\xpos,\ypos)(-1,-1)[\tokena``{#4}]{\height}{\arrowtypea}l%
\putvmorphism(\xpos,\ypos)[``{#5}]{\height}{\arrowtypeb}m%
\putmorphism(\xpos,\ypos)(1,-1)[``{#6}]{\height}{\arrowtypec}r%
}}
 
\def\putAtrianglepair{\@ifnextchar <{\putAtrianglepairp}{\putAtrianglepairp%
   <\arrowtypea`\arrowtypeb`\arrowtypec`\arrowtyped`\arrowtypee;\height>}}
\def\Atrianglepair{\@ifnextchar <{\Atrianglepairp}{\Atrianglepairp%
   <\arrowtypea`\arrowtypeb`\arrowtypec`\arrowtyped`\arrowtypee;\height>}}
 
\def\Atrianglepairp<#1>[#2;#3`#4`#5`#6`#7]{{%
\settripairparms[#1]%
\settokens[#2]%
\width=\height
\xext=\width
\yext=\height
\topadjust[\tokena``]%
\vertadjust[\tokenb`\tokenc`{#6}]
\tempcountd=\tempcounta                       
\vertadjust[\tokenc`\tokend`{#7}]
\ifnum\tempcounta<\tempcountd                 
\tempcounta=\tempcountd\fi                    
\advance \yext by\tempcounta                  
\advance \yoff by-\tempcounta                 %
\multiply \xext by2 
\leftsladjust[\tokenb`\tokena`{#3}]
\rightsladjust[\tokend`\tokena`{#5}]%
\begin{picture}(\xext,\yext)(\xoff,\yoff)%
\putAtrianglepairp
<\arrowtypea`\arrowtypeb`\arrowtypec`\arrowtyped`\arrowtypee;\height>%
(0,0)[#2;#3`#4`#5`#6`{#7}]%
\end{picture}%
}}
 
\def\putVtrianglep<#1>(#2,#3)[#4`#5`#6;#7`#8`#9]{{%
\settriparms[#1]%
\xpos=#2 \ypos=#3
\advance\ypos by\height
{\multiply\height by2
\puthmorphism(\xpos,\ypos)[#4`#5`{#7}]{\height}{\arrowtypea}a}%
\putmorphism(\xpos,\ypos)(1,-1)[`#6`{#8}]{\height}{\arrowtypeb}l%
\advance\xpos by\height
\advance\xpos by\height
\putmorphism(\xpos,\ypos)(-1,-1)[``{#9}]{\height}{\arrowtypec}r%
}}

\def\putVtriangle{\@ifnextchar <{\putVtrianglep}{\putVtrianglep
   <\arrowtypea`\arrowtypeb`\arrowtypec;\height>}}
\def\Vtriangle{\@ifnextchar <{\Vtrianglep}{\Vtrianglep
   <\arrowtypea`\arrowtypeb`\arrowtypec;\height>}}
 
\def\Vtrianglep<#1>[#2`#3`#4;#5`#6`#7]{{
\settriparms[#1]
\width=\height                         
\xext=\width                           
\yext=\height                          
\topadjust[#2`#3`{#5}]
\botadjust[#4``]
\multiply \xext by2 
\leftsladjust[#2`#3`{#6}]
\rightsladjust[#3`#4`{#7}]
\begin{picture}(\xext,\yext)(\xoff,\yoff)%
\putVtrianglep<\arrowtypea`\arrowtypeb`\arrowtypec;\height>%
(0,0)[#2`#3`#4;#5`#6`{#7}]%
\end{picture}%
}}
 
\def\putVtrianglepairp<#1>(#2)[#3;#4`#5`#6`#7`#8]{{
\settripairparms[#1]%
\setpos(#2)%
\settokens[#3]%
\advance\ypos by\height
\putmorphism(\xpos,\ypos)(1,-1)[`\tokend`{#6}]{\height}{\arrowtypec}l%
\puthmorphism(\xpos,\ypos)[\tokena`\tokenb`{#4}]{\height}{\arrowtypea}a%
\advance\xpos by\height
\putvmorphism(\xpos,\ypos)[``{#7}]{\height}{\arrowtyped}m%
\advance\xpos by\height
\putmorphism(\xpos,\ypos)(-1,-1)[``{#8}]{\height}{\arrowtypee}r%
}}
 
\def\putVtrianglepair{\@ifnextchar <{\putVtrianglepairp}{\putVtrianglepairp%
    <\arrowtypea`\arrowtypeb`\arrowtypec`\arrowtyped`\arrowtypee;\height>}}
\def\Vtrianglepair{\@ifnextchar <{\Vtrianglepairp}{\Vtrianglepairp%
    <\arrowtypea`\arrowtypeb`\arrowtypec`\arrowtyped`\arrowtypee;\height>}}
 
\def\Vtrianglepairp<#1>[#2;#3`#4`#5`#6`#7]{{%
\settripairparms[#1]%
\settokens[#2]
\xext=\height                  
\width=\height                 
\yext=\height                  
\vertadjust[\tokena`\tokenb`{#4}]
\tempcountd=\tempcounta        
\vertadjust[\tokenb`\tokenc`{#5}]
\ifnum\tempcounta<\tempcountd%
\tempcounta=\tempcountd\fi
\advance \yext by\tempcounta
\botadjust[\tokend``]%
\multiply \xext by2
\leftsladjust[\tokena`\tokend`{#6}]%
\rightsladjust[\tokenc`\tokend`{#7}]%
\begin{picture}(\xext,\yext)(\xoff,\yoff)%
\putVtrianglepairp
<\arrowtypea`\arrowtypeb`\arrowtypec`\arrowtyped`\arrowtypee;\height>%
(0,0)[#2;#3`#4`#5`#6`{#7}]%
\end{picture}%
}}

\def\putCtrianglep<#1>(#2,#3)[#4`#5`#6;#7`#8`#9]{{%
\settriparms[#1]%
\xpos=#2 \ypos=#3
\advance\ypos by\height
\putmorphism(\xpos,\ypos)(1,-1)[``{#9}]{\height}{\arrowtypec}l%
\advance\xpos by\height
\advance\ypos by\height
\putmorphism(\xpos,\ypos)(-1,-1)[#4`#5`{#7}]{\height}{\arrowtypea}l%
{\multiply\height by 2
\putvmorphism(\xpos,\ypos)[`#6`{#8}]{\height}{\arrowtypeb}r}%
}}
 
\def\putCtriangle{\@ifnextchar <{\putCtrianglep}{\putCtrianglep
    <\arrowtypea`\arrowtypeb`\arrowtypec;\height>}}
\def\Ctriangle{\@ifnextchar <{\Ctrianglep}{\Ctrianglep
    <\arrowtypea`\arrowtypeb`\arrowtypec;\height>}}
 
\def\Ctrianglep<#1>[#2`#3`#4;#5`#6`#7]{{
\settriparms[#1]
\width=\height                          
\xext=\width                            
\yext=\height                           
\multiply \yext by2 
\topadjust[#2``]
\botadjust[#4``]
\sladjust[#3`#2`{#5}]{\width}
\tempcountd=\tempcounta                 
\sladjust[#3`#4`{#7}]{\width}
\ifnum \tempcounta<\tempcountd          
\tempcounta=\tempcountd\fi              
\advance \xext by\tempcounta            
\advance \xoff by-\tempcounta           %
\rightadjust[#2`#4`{#6}]
\begin{picture}(\xext,\yext)(\xoff,\yoff)%
\putCtrianglep<\arrowtypea`\arrowtypeb`\arrowtypec;\height>%
(0,0)[#2`#3`#4;#5`#6`{#7}]%
\end{picture}%
}}
 
\def\putDtrianglep<#1>(#2,#3)[#4`#5`#6;#7`#8`#9]{{%
\settriparms[#1]%
\xpos=#2 \ypos=#3
\advance\xpos by\height \advance\ypos by\height
\putmorphism(\xpos,\ypos)(-1,-1)[``{#9}]{\height}{\arrowtypec}r%
\advance\xpos by-\height \advance\ypos by\height
\putmorphism(\xpos,\ypos)(1,-1)[`#5`{#8}]{\height}{\arrowtypeb}r%
{\multiply\height by 2
\putvmorphism(\xpos,\ypos)[#4`#6`{#7}]{\height}{\arrowtypea}l}%
}}
 
\def\putDtriangle{\@ifnextchar <{\putDtrianglep}{\putDtrianglep
    <\arrowtypea`\arrowtypeb`\arrowtypec;\height>}}
\def\Dtriangle{\@ifnextchar <{\Dtrianglep}{\Dtrianglep
   <\arrowtypea`\arrowtypeb`\arrowtypec;\height>}}
 
\def\Dtrianglep<#1>[#2`#3`#4;#5`#6`#7]{{
\settriparms[#1]
\width=\height                         
\xext=\height                          
\yext=\height                          
\multiply \yext by2 
\topadjust[#2``]
\botadjust[#4``]
\leftadjust[#2`#4`{#5}]
\sladjust[#3`#2`{#5}]{\height}
\tempcountd=\tempcountd                
\sladjust[#3`#4`{#7}]{\height}
\ifnum \tempcounta<\tempcountd         
\tempcounta=\tempcountd\fi             
\advance \xext by\tempcounta           %
\begin{picture}(\xext,\yext)(\xoff,\yoff)
\putDtrianglep<\arrowtypea`\arrowtypeb`\arrowtypec;\height>%
(0,0)[#2`#3`#4;#5`#6`{#7}]%
\end{picture}%
}}
 
\def\setrecparms[#1`#2]{\width=#1 \height=#2}%
\def\recursep<#1`#2>[#3;#4`#5`#6`#7`#8]{{%
\width=#1 \height=#2
\settokens[#3]
\settowidth{\tempdimen}{$\tokena$}
\ifdim\tempdimen=0pt
  \savebox{\tempboxa}{\hbox{$\tokenb$}}%
  \savebox{\tempboxb}{\hbox{$\tokend$}}%
  \savebox{\tempboxc}{\hbox{$#6$}}%
\else
  \savebox{\tempboxa}{\hbox{$\hbox{$\tokena$}\times\hbox{$\tokenb$}$}}%
  \savebox{\tempboxb}{\hbox{$\hbox{$\tokena$}\times\hbox{$\tokend$}$}}%
  \savebox{\tempboxc}{\hbox{$\hbox{$\tokena$}\times\hbox{$#6$}$}}%
\fi
\ypos=\height
\divide\ypos by 2
\xpos=\ypos
\advance\xpos by \width
\xext=\xpos \yext=\height
\topadjust[#3`\usebox{\tempboxa}`{#4}]%
\botadjust[#5`\usebox{\tempboxb}`{#8}]%
\sladjust[\tokenc`\tokenb`{#5}]{\ypos}%
\tempcountd=\tempcounta
\sladjust[\tokenc`\tokend`{#5}]{\ypos}%
\ifnum \tempcounta<\tempcountd
\tempcounta=\tempcountd\fi
\advance \xext by\tempcounta
\advance \xoff by-\tempcounta
\rightadjust[\usebox{\tempboxa}`\usebox{\tempboxb}`\usebox{\tempboxc}]%
\bfig
\putCtrianglep<-1`1`1;\ypos>(0,0)[`\tokenc`;#5`#6`{#7}]%
\puthmorphism(\ypos,0)[\tokend`\usebox{\tempboxb}`{#8}]{\width}{-1}b%
\puthmorphism(\ypos,\height)[\tokenb`\usebox{\tempboxa}`{#4}]{\width}{-1}a%
\advance\ypos by \width
\putvmorphism(\ypos,\height)[``\usebox{\tempboxc}]{\height}1r%
\efig
}}
 
\def\recurse{\@ifnextchar <{\recursep}{\recursep<\width`\height>}}
 
\def\puttwohmorphisms(#1,#2)[#3`#4;#5`#6]#7#8#9{{%
\puthmorphism(#1,#2)[#3`#4`]{#7}0a
\ypos=#2
\advance\ypos by 20
\puthmorphism(#1,\ypos)[\phantom{#3}`\phantom{#4}`#5]{#7}{#8}a
\advance\ypos by -40
\puthmorphism(#1,\ypos)[\phantom{#3}`\phantom{#4}`#6]{#7}{#9}b
}}
 
\def\puttwovmorphisms(#1,#2)[#3`#4;#5`#6]#7#8#9{{%
\putvmorphism(#1,#2)[#3`#4`]{#7}0a
\xpos=#1
\advance\xpos by -20
\putvmorphism(\xpos,#2)[\phantom{#3}`\phantom{#4}`#5]{#7}{#8}l
\advance\xpos by 40
\putvmorphism(\xpos,#2)[\phantom{#3}`\phantom{#4}`#6]{#7}{#9}r
}}
 
\def\puthcoequalizer(#1)[#2`#3`#4;#5`#6`#7]#8#9{{%
\setpos(#1)%
\puttwohmorphisms(\xpos,\ypos)[#2`#3;#5`#6]{#8}11%
\advance\xpos by #8
\puthmorphism(\xpos,\ypos)[\phantom{#3}`#4`#7]{#8}1{#9}
}}
 
\def\putvcoequalizer(#1)[#2`#3`#4;#5`#6`#7]#8#9{{%
\setpos(#1)%
\puttwovmorphisms(\xpos,\ypos)[#2`#3;#5`#6]{#8}11%
\advance\ypos by -#8
\putvmorphism(\xpos,\ypos)[\phantom{#3}`#4`#7]{#8}1{#9}
}}
 
\def\putthreehmorphisms(#1)[#2`#3;#4`#5`#6]#7(#8)#9{{%
\setpos(#1) \settypes(#8)
\if a#9 %
     \vertsize{\tempcounta}{#5}%
     \vertsize{\tempcountb}{#6}%
     \ifnum \tempcounta<\tempcountb \tempcounta=\tempcountb \fi
\else
     \vertsize{\tempcounta}{#4}%
     \vertsize{\tempcountb}{#5}%
     \ifnum \tempcounta<\tempcountb \tempcounta=\tempcountb \fi
\fi
\advance \tempcounta by 60
\puthmorphism(\xpos,\ypos)[#2`#3`#5]{#7}{\arrowtypeb}{#9}
\advance\ypos by \tempcounta
\puthmorphism(\xpos,\ypos)[\phantom{#2}`\phantom{#3}`#4]{#7}{\arrowtypea}{#9}
\advance\ypos by -\tempcounta \advance\ypos by -\tempcounta
\puthmorphism(\xpos,\ypos)[\phantom{#2}`\phantom{#3}`#6]{#7}{\arrowtypec}{#9}
}}
 
\def\putarc(#1,#2)[#3`#4`#5]#6#7#8{{%
\xpos #1
\ypos #2
\width #6
\arrowlength #6
\putbox(\xpos,\ypos){#3\vphantom{#4}}%
{\advance \xpos by\arrowlength
\putbox(\xpos,\ypos){\vphantom{#3}#4}}%
\horsize{\tempcounta}{#3}%
\horsize{\tempcountb}{#4}%
\divide \tempcounta by2
\divide \tempcountb by2
\advance \tempcounta by30
\advance \tempcountb by30
\advance \xpos by\tempcounta
\advance \arrowlength by-\tempcounta
\advance \arrowlength by-\tempcountb
\halflength=\arrowlength \divide\halflength by 2
\divide\arrowlength by 5
\put(\xpos,\ypos){\bezier{\arrowlength}(0,0)(50,50)(\halflength,50)}
\ifnum #7=-1 \put(\xpos,\ypos){\vector(-3,-2)0} \fi
\advance\xpos by \halflength
\put(\xpos,\ypos){\xpos=\halflength \advance\xpos by -50
   \bezier{\arrowlength}(0,50)(\xpos,50)(\halflength,0)}
\ifnum #7=1 {\advance \xpos by
   \halflength \put(\xpos,\ypos){\vector(3,-2)0}} \fi
\advance\ypos by 50
\vertsize{\tempcounta}{#5}%
\divide\tempcounta by2
\advance \tempcounta by20
\if a#8 %
   \advance \ypos by\tempcounta
   \putbox(\xpos,\ypos){#5}%
\else
   \advance \ypos by-\tempcounta
   \putbox(\xpos,\ypos){#5}%
\fi
}}
 
\makeatother

\newtheorem{theorem}{Satz}

\newtheorem{corollary}[theorem]{Folgerung}
\newtheorem{remark}[theorem]{Bemerkung}

\newtheorem{lemma}[theorem]{Lemma}
\newtheorem{beispiel}{Beispiel}
\newtheorem{bemerkung}[theorem]{Bemerkung}

\newcommand{\moo}{{M^{\circ\circ}}}
\newcommand{\uoo}{{U^{\circ\circ}}}

\newcommand{\py}{{\mathfrak{p}}}
\newcommand{\my}{{\mathfrak{m}}}
\newcommand{\qy}{{\mathfrak{q}}}

\newcommand{\cC}{{\mathcal {C}}}

\newenvironment{proof}{{\it Beweis}. $\;\;$}{\hspace*{\fill} $\Box$}

\begin{document}

\title{\Large\bf "Uber rein-wesentliche Erweiterungen}

\author{Helmut Z"oschinger\\ 
Mathematisches Institut der Universit"at M"unchen\\
Theresienstr. 39, D-80333 M"unchen, Germany\\
E-mail: zoeschinger@mathematik.uni-muenchen.de}

\date{}
\maketitle

\vspace{1cm}

\begin{center}  
{\bf Abstract}
\end{center}

Let $(R,\my)$ be a noetherian local ring and let $\cC$ be the class of all $R$-modules $M$ which possess a reflexive submodule $U$
such that $M/U$ is finitely generated.
For every $R$-module $M \in \cC$ the canonical embedding $\varphi: M \to \moo$ is pure-essential.
We investigate in the first section under which conditions the reverse is true, for example if $R$ is a discrete valuation ring or if $R$ 
does not have nilpotent elements and $M$ is flat.
In section 2 we determine all reflexive and flat $R$-modules with the help of a certain analogy between the localization $R_\qy$ and the injective hull of $R/\qy$.
In section 3 we show: If  the property 'pure-essential' is  transitive for a domain $R$, then it follows that $\dim(R)\leq 1.$

\vspace{0.5cm}


\vspace{0.5cm}



\section{Quasi-reflexive Moduln}

Stets sei $(R,\my)$ ein noetherscher lokaler Ring, $E$ die injektive H"ulle seines Restklassenk"orpers $k=R/\my$ und $M^\circ = $ Hom$_R(M,E)$ das Matlis-Duale von $M$.
Wir sagen, $M$ sei {\it quasi-reflexiv}, wenn die kanonische Einbettung $\varphi: M \to \moo$ rein-wesentlich ist, d.h. aus $X \subset \moo,\;X \cap $ Bi $\varphi =0$ und $(X \oplus $ Bi $\varphi)/X$ rein in $\moo/X$ stets folgt $X=0$.
Jeder endlich erzeugte $R$-Modul $M$ ist quasi-reflexiv, denn bekanntlich ist $M \subset \hat{M}$ rein-wesentlich, wegen $\hat{M} \cong \moo$ also auch $\varphi$.
Wir fragen uns, wann sich ein quasi-reflexiver $R$-Modul aus einem reflexiven und einem endlich erzeugten $R$-Modul zusammensetzt.

\begin{lemma}
Sei $M$ ein $R$-Modul und $U$ ein Untermodul von $M$.
\begin{enumerate}
\item[(a)]
Ist $M$ quasi-reflexiv und $U$ rein in $M$, so ist auch $U$ quasi-reflexiv.
\item[(b)]
Ist $U$ reflexiv und $M/U$ quasi-reflexiv, so ist auch $M$ quasi-reflexiv.
\item[(c)]
Ist $U$ reflexiv und $M/U$ sogar endlich erzeugt, so ist auch $P(M)$ reflexiv und $M/P(M)$ endlich erzeugt.
\end{enumerate}
\end{lemma}

\begin{proof}
\begin{enumerate}
\item[(a)]
Seien $\varphi: M \to \moo$ und $\chi : U \to \uoo$ die kanonischen Einbettungen und sei $X \subset \uoo,\; X \cap $ Bi $ \chi=0,\;(X \oplus $ Bi $ \chi)/X$  rein in $\uoo/X$. Im kommutativen Diagramm 

\vspace{0.5cm}

$$ 0\;\longrightarrow\quad U\qquad\subset \qquad M\quad\longrightarrow\quad M/U\;\longrightarrow \;0$$

\hspace{4.2cm} $\downarrow \;\chi'$ \hspace{1.6cm} $\downarrow\;\varphi'$\hspace{2.0cm} $\downarrow\;\psi$
$$0\;\longrightarrow\;\uoo/X\;\longrightarrow\;\moo/\alpha(X)\;\longrightarrow \;(M/U)^{\circ\circ}\;\longrightarrow \;0$$

\vspace{0.5cm}

sind dann beide Zeilen rein-exakt (die untere mit dem von $U \subset M$ induzierten Monomorphismus $\alpha: \uoo  \to \moo$ sogar zerfallend), so da"s mit $\chi'$ und $\psi$  auch $\varphi'$ ein reiner Monomorphismus  ist, denn f"ur jeden $R$-Modul $A$ sind $1_A \otimes \chi'$ und $1_A \otimes \psi$ injektiv, also auch $1_A \otimes  \varphi'$.
Weil $M$ quasi-injektiv war, folgt $\alpha(X)=0,\;X=0$ wie verlangt.
\item[(b)]
Wir zeigen im {\it 1.Schritt} f"ur jede reine Erweiterung $A \subset B:$ Ist $U \subset A$ und $A/U$ rein-wesentlich in $B/U$, so auch $A$  in $B$.
Zum Beweis sei $X \subset B,\;X \cap  A=0$ und $(X \oplus A)/X$ rein in $B/X$.
Dann ist auch $(X+A)/(X+U)$ rein in $B/(X+U)$, mit $ \overline{B}=B/U$ also $\overline{X} \cap \overline{A}=0$ und $(\overline{X}\oplus  \overline{A})/\overline{X}$  rein in $\overline{B}/\overline{X}$.
Nach Voraussetzung folgt $\overline{X} =0,\;X \subset U,\;X=0$ wie verlangt.

Seien im {\it 2.Schritt} $U$ und $M/U$ wie angegeben. Im kommutativen Diagramm

\vspace{0.5cm}

$$ 0\;\longrightarrow\;U\qquad\subset \qquad M\;\longrightarrow\;M/U\;\longrightarrow \;0$$

\hspace{4.2cm} $\cong\;\downarrow \;\chi$ \hspace{1.6cm} $\downarrow\;\varphi$\hspace{1.2cm} $\downarrow\;\psi$
$$0\;\longrightarrow\quad\uoo\quad \stackrel{\alpha}{\longrightarrow}\quad \moo\;\stackrel{\beta}{\longrightarrow} \;(M/U)^{\circ\circ}\;\longrightarrow \;0$$

\vspace{0.5cm}

ist dann Bi $\alpha \subset $ Bi $ \varphi \subset \moo$ und Bi $\varphi /$ Bi $ \alpha$ nach Voraussetzung rein-wesentlich  in $\moo/$ Bi $ \alpha$, also nach dem ersten Schritt Bi $\varphi$ rein-wesentlich in $\moo$.
\item[(c)]
F"ur jeden $R$-Modul $M$ sei $P(M)$ der gr"o"ste  radikalvolle Untermodul von $M$. Wir zeigen  im {\it 1.Schritt} f"ur jeden reflexiven $R$-Modul $A$, da"s $A/P(A)$  endlich erzeugt ist: Ein  Komplement $V$ von $\my A$ in $A$ ist endlich erzeugt, ein Komplement $W$ von $V$ in $A$  ist radikalvoll, so da"s $W \subset P(A)$ die Behauptung liefert.
Sind im {\it 2.Schritt} $U$  und $M/U$wie angegeben, folgt aus  $P(M/U)=0,\;P(M)\subset U$ die Reflexivit"at von $P(M)$, und  in der exakten Folge
$$0\;\longrightarrow \frac{U}{P(M)}\;\subset\; \frac{M}{P(M)}\;\longrightarrow  \;\frac{M}{U}\;\longrightarrow \;0$$
ist nach dem ersten Schritt $U/P(M)$ endlich erzeugt, also auch $M/P(M)$.
\end{enumerate}
\end{proof}

Sei jetzt $\cC$ die Klasse aller $R$-Moduln  $M$, bei denen $P(M)$ reflexiv und $M/P(M)$ endlich erzeugt ist. Punkt (b) des Lemmas  liefert sofort

\begin{corollary}
Jeder $R$-Modul $M \in \cC$ ist quasi-reflexiv.
\end{corollary}

Die anfangs gestellte Frage lautet jetzt genauer:
Wann geh"ort ein quasi-reflexiver$R$-Modul $M$ zur Klasse $\cC$?
Teilantworten dazu geben wir in den Beispielen 1 und 2 sowie in (1.7) und (1.9).

\begin{lemma}
Die Klasse $\cC$ ist gegen"uber Untermoduln, Faktormoduln und Gruppenerweiterungen abgeschlossen.
\end{lemma}

\begin{proof}
Ist $M \in \cC$ und $U \subset M$, so ist $U \cap P(M)$ reflexiv und $U/U\cap P(M)$ endlich erzeugt, ebenso $(U+P(M))/U$ reflexiv und $M/(U+P(M))$ endlich erzeugt, nach (1.1c) also $U \in \cC,\;M/U \in \cC.$
Zur Umkehrung seien jetzt $M$ beliebig und $U,\;M/U \in \cC:$ Dann ist $M$ minimax und $R/\py$ vollst"andig f"ur alle $\py \in $ Koass$(M) \subset $ Koass$(U)\, \cup $ Koass$(M/U)$, also nach \cite[Satz 2.8]{15} $P(M)$ reflexiv und $M/P(M)$ koatomar, $M \in \cC$.
\end{proof}

\begin{bemerkung}
{\rm Definiert man $\cC'$ als die Klasse aller $R$-Moduln $M$, bei denen $P(M)$ reflexiv und $M/P(M)$ nur koatomar ist, so zeigt der zitierte Satz 2.8, da"s $M$ genau dann zu $\cC'$ geh"ort, wenn f"ur alle Faktormoduln $M/V$ gilt: Ass$(M/V)\,=$ Ass$((M/V)^{\circ\circ})$.}
\end{bemerkung}

\begin{beispiel}
Sei $M$ quasi-reflexiv und von der Form $M \cong X^{(I)}$ mit $X$ endlich erzeugt. Dann ist $M$ bereits endlich erzeugt.
\end{beispiel}

\begin{proof}
$M$ ist im Sinne von \cite[p.7]{16} totalsepariert, also die rein-injektive H"ulle $N$ nach dem dortigen Lemma 2.1 separiert.
Damit ist $\moo \cong N$ separiert, $M^\circ$ halbartinsch, $M$ koatomar. W"are $X \neq 0$ und $I$ nicht endlich, folgte Ass$(X)=\{\my\}$, denn f"ur jedes $\py \subsetneq \my$ ist $(R/\py)^{({\Bbb N})}$ nicht koatomar.
Also ist $X$ halbartinsch, ja sogar von endlicher L"ange und deshalb nach Zimmermann \cite[Beispiel 2.6.1]{8} $\Sigma$-rein-injektiv.
Damit w"are $M$ rein-injektiv, $\varphi : M \to \moo$ ein Isomorphismus, $M$ reflexiv, $I$ endlich entgegen der Annahme.
\end{proof}

\begin{beispiel}
Sei $M$ quasi-reflexiv und $R$ ein diskreter Bewertungsring. Dann folgt $M \in \cC.$
\end{beispiel}

\begin{proof}
{\it 1.Schritt:} $\;$ Ist $M$ ein $R$-Modul und $M/\my M$ endlich erzeugt, folgt $M=M_1 \oplus M_2\oplus M_3$ mit $M_1$ teilbar, $M_2$ separiert und torsionsfrei, $M_3$ von endlicher L"ange.
Zum Beweis sei $M_3$ ein Basis-Untermodul von $T(M)$.
Dann ist auch $M_3/\my \cdot M_3$ endlich erzeugt, also $M_3$ von endlicher L"ange, $T(M)=D_1 \oplus M_3$ mit $D_1$ teilbar, $M=T(M)\oplus M_0$ mit $M_0$ torsionsfrei, also $M_0=D_2 \oplus M_2$ mit $D_2$ teilbar und $M_2$ separiert.
Mit $M_1=D_1 \oplus D_2$ folgt die Behauptung.

{\it 2.Schritt} $\;$ Sei jetzt $M$ quasi-reflexiv. Nach Fuchs, Salce und Zanardo \cite[Lemma 1.5]{2} ist dann $\moo/M$ teilbar, also $M/\my M$ nach \cite[Folgerung 1.5]{14} reflexiv, d.h. endlich erzeugt und deshalb $M=M_1 \oplus M_2 \oplus M_3$ wie im ersten Schritt.
Alle drei Summanden sind nach (1.1a) wieder quasi-reflexiv: $M_1=D(M)$ ist damit schon reflexiv und $M_2$ sogar totalsepariert, also wie im Beweis von Beispiel 1 koatomar, also wegen torsionsfrei schon endlich erzeugt und frei.
(Hier ist also jeder quasi-reflexive $R$-Modul bereits reflexiv oder endlich erzeugt.)
\end{proof}

\begin{lemma}
Sei $R$ wieder beliebig, $A \subset B$ rein-wesentlich und $Z(B) \subset A$.
Dann folgt f"ur jeden injektiven Untermodul $X$ von $B$ bereits $X \subset A.$
\end{lemma}

\begin{proof}
F"ur jeden $R$-Modul $M$ sei $Z(M)$ der singul"are Untermodul von $M$. Mit $A_1/A =Z(B/A)$ ist dann $A$ gro"s in $A_1$, also wegen der Reinheit $A=A_1$, d.h. $Z(B/A)=0.$

Den injektiven Modul $X \subset B$ k"onnen wir gleich uniform annehmen. W"are $X \not\subset A$, also $0\neq X/X\cap A \cong (X+A)/A \subset B/A$, folgte aus $Z(X/X \cap A) =0$  bereits $X \cap A =0,\;(X \oplus A)/A$ direkter Summand in $B/A$, also auch $(X \oplus A)/X$ rein in $B/X$, und damit der Widerspruch $X=0.$
\end{proof}

\begin{corollary}
Sei $R$ ohne nilpotente Elemente, $A \subset B$ rein-wesentlich und $B$ torsionsfrei. Dann folgt $D(B) \subset A$.
\end{corollary}

\begin{proof}
Weil $R$ keine nilpotenten Elemente hat, ist bekanntlich jeder teilbare und torsionsfreie $R$-Modul bereits injektiv und flach. Wegen $Z(B)=T(B)=0$ folgt mit (1.5) die Behauptung. 
\end{proof}

\begin{theorem}
Sei $R$ ohne nilpotente Elemente, $M$ quasi-reflexiv und torsionsfrei. Dann gilt:
\begin{enumerate}
\item[(a)]
$M$ hat endliche Goldie-Dimension.
\item[(b)]
$D(M)$ ist reflexiv.
\item[(c)]
$M/D(M)$ ist kotorsion und separiert.
\end{enumerate}
\end{theorem}

\begin{proof}
(b)  Weil der divisible Anteil $D(M)$ teilbar und torsionsfrei, also injektiv ist, folgt $M=D(M) \oplus M_1.$ Beide Summanden sind nach (1.1a) wieder quasi-reflexiv, also $D(M)$ bereits reflexiv.

F"ur (a) und (c) sei jetzt gleich $D(M)=0$. Nach (1.6) folgt dann $D(\moo)=0$, also $M^\circ$ torsion, so da"s alle $\py \in $ Koass$(M)= $ Ass$(M^\circ)$ regul"ar sind, d.h. $M$ im Sinne von \cite{14} kotorsion ist.
Satz 2.1 dieser Arbeit sagt dann, da"s $M$ von endlicher Goldie-Dimension ist, und Bemerkung 2.2, da"s $M$ in seiner injektiven H"ulle $E(M)$ klein ist.
Weil aber $E(M)$ teilbar und torsionsfrei, also flach ist, folgt nach \cite[Satz 1.2]{13} $\bigcap\limits_{i\geq 1} \my^i M=0.$
\end{proof}

\begin{bemerkung}
{\rm Sei weiter $R$ ohne nilpotente Elemente. (1) Falls $\my \neq 0$, ist $\my$ bereits regul"ar, in (b) und (c) also $D(M)=P(M).$
(2) Ist $M=D(M) \oplus M_1$ wie im Beweis des Satzes und hat Koass$(M_1)$ eine endliche finale Teilmenge (z.B. wenn $\dim (R)\leq 1$ ist, oder $M_1$ minimax, oder $M_1$ flach \cite[Satz 2.1]{13}),
so ist $M_1$ im Sinne von \cite{14} sogar stark kotorsion, also nach dem dortigen Satz 2.3 endlich erzeugt. Speziell gilt:}
\end{bemerkung}

\begin{corollary}
Sei $R$ ohne nilpotente Elemente, $M$ quasi-reflexiv und flach. Dann folgt $M \in \cC$.
\end{corollary}

\section{Die reflexiven flachen Moduln}
\setcounter{theorem}{0}

Ist $R$ ein lokaler Integrit"atsring mit Quotientenk"orper $K \neq R$, so folgt aus (1.7) und (1.8) sofort die Struktur aller reflexiven torsionsfreien $R$-Moduln:
Sie sind von der Form $K^n \times M_1$ mit $M_1$ endlich erzeugt und vollst"andig, und $K$ ist genau dann reflexiv, wenn $R$ eindimensional und vollst"andig ist.
Hat aber $R$ Nullteiler, mu"s man $K$ durch bestimmte Lokalisierungen $R_\qy\;(\qy \subsetneq \my)$ ersetzen (Satz 2.6), und es zeigt sich eine bemerkenswerte Parallelit"at zwischen $R_\qy$ und $E(R/\qy)$, der wir in den ersten drei Lemmata nachgehen.

\begin{lemma}
F"ur ein Primideal $\qy \subsetneq \my$ sind "aquivalent:
\begin{enumerate}
\item[(i)]
$M=R_\qy$ besitzt eine Koprim"arzerlegung.
\item[(ii)]
$N=E(R/\qy)$ besitzt eine Prim"arzerlegung.
\item[(iii)]
$h(\qy)=0.$
\item[(iv)]
$M$ ist rein-injektiv.
\end{enumerate}
\end{lemma}

\begin{proof}\quad\\
$(i\leftrightarrow iii)\;$ Weil alle Elemente von $S=R\backslash \qy$ auf $M$ bijektiv operieren, hat $M$ genau dann eine Koprim"arzerlegung als $R$-Modul, wenn $M_S$ eine als $R_S$-Modul hat \cite[Lemma 1.7]{11}, d.h.
 wenn der Ring $R_\qy$ artinsch ist.

$(ii \leftrightarrow iii)\;$ Ein beliebiger injektiver $R$-Modul $B$ hat nach \cite[Lemma 3.5]{12} genau dann eine Prim"arzerlegung, wenn Ass$(B) \subset $ Min$(R)$ ist.

$(iii \leftrightarrow iv)\;$ Nach einem Satz von F.K.Schmidt (siehe \cite[Satz 2.4.1]{5}) gilt f"ur jeden noetherschen lokalen Integrit"atsring $R$: Ist $R \subset T \subsetneq K$ ein Oberring und $T$ lokal und henselsch, so ist $T$ ganz "uber $R$.
Insbesondere ist f"ur jedes Primideal $0 \neq \qy \subsetneq \my$ der lokale Integrit"atsring $R_\qy$ {\it nicht} vollst"andig.
Auch wenn $R$ Nullteiler hat, gilt deshalb f"ur jedes Primideal $\qy \subsetneq \my$, da"s Koass$_{R_\qy}(\widehat{R_\qy}) = $ Spec$(R_\qy)$ ist \cite[Satz 2.1]{13}, also $R_\qy$ als lokaler Ring nur dann vollst"andig ist, wenn $h(\qy)=0$ ist (vgl. \cite[Proposition 11.5]{4} oder \cite[Lemma 2.1]{3}).

Weil nun $M=R_\qy$ genau dann als $R$-Modul rein-injektiv ist, wenn es $M_S$ als $R_S$-Modul ist \cite[Theorem 7.1]{4}, folgt die Behauptung.
\end{proof}

\begin{lemma}
F"ur ein Primideal $\qy \subsetneq \my$ sind "aquivalent:
\begin{enumerate}
\item[(i)]
$M=R_\qy$ ist ein Minimax-Modul.
\item[(ii)]
$N=E(R/\qy)$ ist ein Minimax-Modul.
\item[(iii)]
$h(\qy)=0$ und $\dim(R/\qy)=1.$
\end{enumerate}
\end{lemma}

\begin{proof}\quad\\
$(i \to iii)\;$ Jeder  Minimax-Modul erf"ullt die Minimalbedingung f"ur radikalvolle Untermoduln, so da"s aus $\qy^e M=\qy^{e+1} M=\cdots\;$ f"ur ein $e\geq 1$ folgt $h(\qy)=0.$
Weil $M/\qy M \cong \kappa(\qy)$ auch als $R/\qy$-Modul minimax ist, gilt nach \cite[Lemma 1.1]{9} $\dim(R/\qy)=1.$

$(iii \to i)\;$ Wegen $\dim(R/\qy)=1$ ist $\kappa(\qy)$ als $R/\qy$-Modul, also auch als $R$-Modul minimax.
Mit $M/\qy M$ sind dann auch alle $M/\qy^n M\;(n\geq 2)$ minimax, wegen $h(\qy)=0$ also schon $M$ selbst.

$(ii \leftrightarrow iii)\;$ Ein beliebiger injektiver $R$-Modul $B$ ist nach \cite[Folgerung 2.5]{10}
genau dann minimax, wenn $B$ endliche Goldie-Dimension hat und f"ur jedes $\py \in $ Ass$(B/L(B))$ gilt: $h(\py)=0$ und $\dim(R/\py) =1.$
\end{proof}

\begin{lemma}
F"ur ein Primideal $\qy \subsetneq \my$ sind "aquivalent:
\begin{enumerate}
\item[(i)]
$M=R_\qy$ ist reflexiv.
\item[(ii)]
$N=E(R/\qy)$ ist reflexiv.
\item[(iii)]
$h(\qy) =0$ und $R/\qy$ ist eindimensional und vollst"andig.
\item[(iv)]
$M \cong N^\circ.$
\item[(v)]
$N \cong M^\circ$.
\end{enumerate}
\end{lemma}

\begin{proof}
Ein beliebiger $R$-Modul $A$ ist genau dann reflexiv, wenn $A$ minimax und $R/$Ann$_R(A)$ vollst"andig ist (siehe \cite[Lemma 1.1]{14} oder \cite[Theorem 12]{1}).
Damit folgen die ersten drei "Aquivalenzen sofort aus (2.2):
Bei $(i \to iii)$ ist $\qy \in $ Ass$(M^\circ)$, bei $(ii \to iii)\;\, \qy \in $ Ass$(N)$, also beide Male $R/\qy$ als Untermodul eines reflexiven Moduls vollst"andig.
Bei $(iii \to i)$ fogt aus $h(\qy)=0$, da"s $\qy^e \subset $ Ann$_R(M)$ ist f"ur ein $e\geq 1$, bei $(ii \to i)$ entsprechend $\qy^e \subset $ Ann$_R(N),$ also beide Male die gew"unschte Reflexivit"at.

Bei $(i \to iv)$ gen"ugt $M \cong A^\circ$ mit irgend einem $R$-Modul $A$. Dann ist n"amlich $A$ injektiv und unzerlegbar, $A \cong E(R/\py)$ f"ur ein $\py \in $ Spec$(R)$,
wegen $\py = \bigcup $ Koatt$(A) = \bigcup $ Att$(M)=\qy$ also $A \cong N.$
Bei $(ii \to v)$ gen"ugt $N \cong B^\circ$ mit irgend einem $R$-Modul $B$, denn dann ist $B$ flach, unzerlegbar und Koass$(B)=\{\qy\},$ also nach \cite[Beispiel 2.6]{13} $B \cong M$.

Bei $(v \to iii)$ gilt stets Ass$(M^\circ)=$ Koass$(M)=\{\py \in $ Spec$(R)\,|\,\py \subset \qy\}$, so da"s in
$$M^\circ \;\cong\; \coprod_{\py \subset \qy} E(R/\py)^{(I_\py)}$$
alle $I_\py\neq \emptyset$ sind und $(M/\qy M)^\circ \cong M^\circ [\qy] \cong \kappa(\qy)^{(I_\qy)},$
also $|I_\qy|\,=\dim_{\kappa(\qy)}(\kappa(\qy)^\Box)$ ist (wobei $\Box$ das Matlis-Duale "uber $R/\qy$ sei).
Allein daraus, da"s $M^\circ$ unzerlegbar ist, folgt jetzt $h(\qy)=0$ und $|I_\qy|=1,$ so da"s $\kappa(\qy)$ als $R/\qy$-Modul reflexiv ist, d.h. $R/\qy$ eindimensional und vollst"andig. Bei $(iv \to iii)$ gilt nach Schenzel \cite[Lemma 2.3]{6} stets
$$N^\circ \;\cong\; \widehat{R_\qy^{(I)}}$$
mit $|I|\,=\dim_{\kappa(\qy)}(\kappa(\qy)^\Box).$
Aus $(iv)$ folgt deshalb $|I|=1$, d.h. wie eben
$R/\qy$ eindimensional und vollst"andig, au"serdem $R_\qy \cong \widehat{R_\qy},$ also mit der "Aquivalenz $(iii \leftrightarrow iv)$ in (2.1) auch noch $h(\qy)=0.$
\end{proof}

\begin{remark}\quad 
{\rm Allein aus der obigen Darstellung von $N^\circ$ folgt eine Versch"arfung des Theorems 1.1 in \cite{6}: Genau dann ist $E(R/\qy)^\circ \cong \widehat{R_\qy}$,
wenn $|I|=1$, d.h. $R/\qy$ eindimensional und vollst"andig ist.}
\end{remark}

\begin{remark}\quad
{\rm Definiert man f"ur einen $R$-Modul $A$ und ein Primideal $\py$ die (0-te) Bass-Zahl $\mu(\py,A)= $ Rang$_{R/\py}(A[\py])$, so ist $A$ nach \cite[Folgerung 1.2]{17}  genau dann reflexiv, wenn $\mu(\py,A)=\mu(\py,A^{\circ\circ})$ gilt f"ur alle $\py \in $ Spec$(R)$.
Im Spezialfall $N=E(R/\qy)$, mit $\qy \subsetneq \my$, ist nun die Bedingung $\mu(\py,N)=\mu(\py,N^{\circ\circ})$ "aquivalent damit,
da"s entweder $\py \not\subset \qy$ ist oder $\py = \qy,\;R/\py$ eindimensional und vollst"andig.
Verlangt man das f"ur alle $\py \in $ Spec$(R)$, erh"alt man einen neuen Beweis f"ur $(ii \leftrightarrow iii)$;
verlangt man das nur f"ur $\py = \qy$, ist $\mu(\qy,N)=\mu(\qy,N^{\circ\circ})$ "aquivalent mit der Situation in (2.4).}
\end{remark}

\begin{theorem}
Ein flacher $R$-Modul $M$ ist genau dann reflexiv, wenn er von der Form
$$M\;\cong\;R_{\py_1}\times \cdots \times R_{\py_n}$$
ist, wobei f"ur jedes $1\leq i \leq n$ gilt:
Falls $\py_i =\my$, ist $R$ vollst"andig; falls $\py_i \subsetneq \my$, ist $h(\py_i)=0$ und $R/\py_i$ eindimensional und vollst"andig.
\end{theorem}

\begin{proof}
Weil jeder reflexive Modul von endlicher Goldie-Dimension, also direkte Summe von endlich vielen unzerlegbaren Moduln ist, k"onnen wir gleich $M$ flach, reflexiv und unzerlegbar annehmen.
Aus $M^\circ$ injektiv und unzerlegbar, $M^\circ \cong E(R/\py)$ folgt Koass$(M)=\{\py\}$, also
nach \cite[Beispiel 2.6]{13} $M\cong R_\py$. Die Zusatzbedingungen sind f"ur $\py=\my$ klar, und f"ur $\py \subsetneq \my$ folgen sie mit (2.3).

Hat umgekehrt $M$ eine Zerlegung wie angegeben, sind alle $R_{\py_i}\;\,(1\leq i\leq n)\;$ wieder nach (2.3) reflexiv, also auch $M$.
\end{proof}

\section{Strikt rein-wesentliche Erweiterungen}
\setcounter{theorem}{0}

Eine Modulerweiterung $A \subset B$ hei"se {\it strikt} rein-wesentlich, wenn $A$ rein in $B$ ist und f"ur jeden Homomorphismus $f: B \to Y$ gilt:
Ist $f|A$ ein reiner Monomorphismus, so auch $f$ (siehe \cite[p.128]{4}).
Das ist "aquivalent damit, da"s $A$ rein-wesentlich in $B$ ist und aus $B \subset Y,\;\,A$ rein in $Y$, stets folgt $B$ rein in $Y$.
"Uber einem diskreten Bewertungsring $R$ ist das nichts Neues, denn aus $A$ rein-wesentlich in $B$ folgt nach \cite[Lemma 5]{2} 
$B/A$ teilbar, also direkter Summand in $Y/A$, und damit $B$ rein in $Y$.
Auch wenn $A$ rein-wesentlich in $B$ ist und $B$ rein-injektiv, ist die Erweiterung nach \cite[Proposition 4.4]{7} strikt rein-wesentlich.
Im allgemeinen ist unsere Bedingung echt st"arker, hat aber den Vorteil, gegen"uber Hintereinanderausf"uhrung  und endlichen Produkten abgeschlossen zu sein.

Ist $M$ ein separierter, flacher $R$-Modul und $F \subset M$ ein Basis-Untermodul  von $M$, so folgt aus $F$ rein in $M,\; H(M)=0$ und $M/F$ radikalvoll, da"s die Erweiterung $F \subset M$ {\it rein-wesentlich} ist \cite[Lemma 1.9]{16}.
F"ur die Striktheit mu"s man mehr verlangen:

\begin{lemma}
Sei $M$ ein separierter, flacher $R$-Modul, $F \subset M$ ein Basis-Untermodul von $M$ und $M \subset N$ eine rein-injektive H"ulle. Dann sind "aquivalent:
\begin{enumerate}
\item[(i)]
Die Erweiterung $F \subset M$ ist strikt rein-wesentlich.
\item[(ii)]
Die Hintereinanderausf"uhrung $F\subset N$ ist rein-wesentlich.
\item[(iii)]
$M$ ist totalsepariert.
\end{enumerate}
\end{lemma}

\begin{proof}
$(i \to ii)\;$ Stets ist $M \subset N$ strikt rein-wesentlich, nach Voraussetzung aber auch $F \subset M$, und dann nat"urlich $F \subset N$.

$(ii \to iii)\;$ Dann ist $N$ auch die rein-injektive H"ulle von $F$, also $N \cong \hat{F}$ separiert, und das ist nach \cite[Lemma 2.1]{16} "aquivalent mit $M$ totalsepariert.

$(iii \to i)\;$ Sei $M \subset Y$ und $F$ rein in $Y$. Im kommutativen Diagramm

\newpage
\vspace{0.5cm}

$$ \hat{F}\;\stackrel{\cong}{\longrightarrow}\quad \hat{M}\;\stackrel{h}{\longrightarrow} \quad Y'$$

\hspace{4.9cm} $\cup $ \hspace{1.4cm} $\cup$\hspace{1.8cm} $\uparrow\;g$
$$\hspace{-0.3cm} F\;\quad\;\subset\quad M\quad\; \subset \quad\;\; Y$$

\vspace{0.5cm}

sei das rechte Quadrat eine Fasersumme, so da"s auch $g$ ein reiner Monomorphismus ist, weil nach Voraussetzung $M \subset \hat{M}$ rein ist. Stets ist $F$ strikt rein-wesentlich in $\hat{F}$, also auch $h$ ein reiner Monomorphismus, und es folgt $M$ rein in $Y$ wie gew"unscht.
\end{proof}

\begin{remark} $\;$
{\rm In \cite[(1.11) und (2.8)]{16} werden separierte flache $R$-Moduln angegeben, die nicht totalsepariert sind, bei denen also insbesondere $(ii)$ verletzt ist. 
Die Frage, ob die Eigenschaft 'rein-wesentlich' transitiv ist, hat eine l"angere Geschichte, die in der Einleitung zu \cite{2} geschildert wird.
Das dortige Hauptergebnis, Theorem 6, besagt: "Uber einem (nicht notwendig lokalen) Integrit"atsring $T$ ist die Eigenschaft 'RD-wesentlich' genau dann transitiv, wenn $T$ ein diskreter Bewertungsring ist.}
\end{remark}

\begin{lemma}
F"ur eine Modulerweiterung $A \subset B$ sind "aquivalent:
\begin{enumerate}
\item[(i)]
$A \subset B$ ist strikt rein-wesentlich.
\item[(ii)]
Ist $B \subset C$ rein-wesentlich, so auch $A \subset C$.
\end{enumerate}
\end{lemma}

\begin{proof}
$(i \to ii)\;$ Sei $X \subset C,\;X \cap A=0$ und $(X \oplus A)/X$ rein in $C/X$.
Mit $f=kan: B \to C/X$ ist dann $f|A$ ein reiner Monomorphismus, also nach Voraussetzung auch $f$, d.h. $X \cap B=0$
und $(X \oplus B)/X$ rein in $C/X$. Es folgt $X=0$.

$(ii \to i)\;$ Weil $B \subset B$ rein-wesentlich ist, ist es nach Voraussetzung auch $A \subset B$. Sei nun $B \subset Y$ und $A$ rein in $Y$.
Mit einer rein-injektiven H"ulle $B \subset N$ und der Fasersumme

\vspace{0.5cm}

$$N\;\stackrel{h}{\longrightarrow}\quad Y'$$
$$\hspace{0.2cm}\cup\qquad\qquad \uparrow\;g$$
$$B\quad\quad \subset \quad\; Y$$

\vspace{0.5cm}

ist nach Voraussetzung auch $A \subset N$ rein-wesentlich, au"serdem $g$, also auch $A \subset N \stackrel{h}{\longrightarrow} Y'\;$ ein reiner Monomorphismus.
Weil $A \subset N$ sogar strikt, also auch $h$ ein reiner Monomorphismus ist, folgt $B$ rein in $Y$.
\end{proof}

\begin{corollary}
Genau dann ist die Eigenschaft 'rein-wesentlich' transitiv, wenn jede rein-wesentliche Erweiterung bereits strikt rein-wesentlich ist.
\end{corollary}

\begin{theorem}
Ist $R$ ein Integrit"atsring und die Eigenschaft 'rein-wesentlich' transitiv, so folgt $\dim(R)\leq 1.$
\end{theorem}

\begin{proof}
Wir zeigen allgemeiner f"ur jeden Ring $R$: Besitzt $R$ ein regul"ares Primideal $\qy \neq \my$, so gibt es eine rein-wesentliche Erweiterung
$R^{(\Bbb N)}\subset B$, die {\it nicht} strikt rein-wesentlich ist (so da"s unter den Voraussetzungen des Satzes mit (3.4) die Behauptung folgt).

Ist n"amlich $M=R^{(\Bbb N)}$ und $\overline{M}$ der $\my$-adische Abschlu"s von $M$ in $P=R^{\Bbb N}$, so leistet $B=M+\qy \cdot \overline{M}$ das gew"unschte.
Mit $\overline{M}/M$ ist auch $\qy \cdot \overline{M}/M=B/M$ radikalvoll, wegen $M$ rein in $B$ und $H(B)=0$ also $M \subset B$ rein-wesentlich \cite[Lemma 1.9]{16}.
Weiter ist $B \subsetneq \overline{M}$, denn mit irgendeinem $s \in \my \backslash \qy$ ist $x=(1,s,s^2,s^3,...) \in \overline{M}$, und
w"are $x \in B$, d.h. $x-y \in \qy \cdot \overline{M}$ f"ur ein $y \in M$, folgte $x-y \in \qy \cdot P$, so da"s fast alle Koeffizienten von $x$ in $\qy$ l"agen,
also der Widerspruch $s \in \qy.$
Erst jetzt brauchen wir einen NNT $r \in \qy$: Mit ihm folgt $rB \subsetneq r \overline{M} \subset B$, d.h. $rB \subsetneq B \cap r \overline{M}$,
so da"s $B$ nicht rein in $\overline{M}$ ist, also $M$ nicht strikt rein-wesentlich in $B$.
\end{proof}

\begin{lemma}
Sind $A_1 \subset B_1$ und $A_2 \subset B_2$ strikt rein-wesentliche Erweiterungen, so auch $A_1 \times A_2\,\subset B_1 \times B_2.$
\end{lemma}

\begin{proof}
Wir zeigen im {\it 1.Schritt} f"ur jedes Tripel $U\subset A \subset B$: Ist $U$ rein in $B$ und $A/U$ strikt rein-wesentlich in $B/U$, so ist auch $A$ strikt rein-wesentlich in $B$.
Klar ist $A$ rein in $B$, also nach dem ersten Beweisschritt von (1.1b) sogar rein-wesentlich in $B$.
F"ur die Striktheit sei jetzt $B \subset Y$ und $A$ rein in $Y$. Aus
$$A/U \;\subset \;B/U\;\subset \;Y/U$$
und $A/U$ rein in $Y/U$ folgt dann nach Voraussetzung $B/U$ rein in $Y/U$, also $B$ rein in $Y$ wie verlangt. Seien im {\it 2.Schritt}
die $A_i \subset B_i\;(i=1,2)\;$ wie angegeben. Die Projektion $B_1 \times A_2 \to B_1$ und der erste Schritt zeigen, da"s $A_1 \times A_2$ strikt rein-wesentlich in $B_1 \times A_2$ ist, entsprechend $B_1 \times A_2$ in $B_1 \times B_2$ und damit folgt die Behauptung.
\end{proof}

Bei einem quasi-reflexiven $R$-Modul $M$ ist die Einbettung $\varphi: M \to \moo$ sogar strikt rein-wesentlich, soda"s man mit (1.1a) und (3.6) erh"alt:

\begin{corollary}
Eine endliche direkte Summe $M_1 \oplus \cdots \oplus M_n$ ist genau dann quasi-reflexiv, wenn es alle $M_i$ sind $\;(1\leq i\leq n)$.
\end{corollary}

\end{document}